\documentclass[10pt,a4paper,twoside]{amsart}
\usepackage{graphicx}
\usepackage{a4wide}

\newtheorem{theorem}{Theorem}[section]
\newtheorem{lemma}{Lemma}[section]

\newtheorem*{maintheorem*}{Main Theorem}
\allowdisplaybreaks
\numberwithin{equation}{section}

\newcommand{\norm}[1]{\left\| #1 \right\|}

\newcommand{\Fed}{F_{\delta}}
\newcommand{\eps}{\varepsilon}

\newcommand{\ed}{{\delta}}
\newcommand{\ued}{u_\ed}
\newcommand{\Ped}{P_\ed}

\newcommand{\ue}{u_\eps}
\newcommand{\dk}{\delta_k}
\newcommand{\uedk}{u_{\dk}}

\newcommand{\Pe}{P_\eps}

\newcommand{\Pedk}{P_{\dk}}

\newcommand{\pt}{\partial_t}

\newcommand{\px}{\partial_x }
\newcommand{\pxx}{\partial_{xx}^2}
\newcommand{\pxxx}{\partial_{xxx}^3}
\newcommand{\pxxxx}{\partial_{xxxx}^4}

\newcommand{\ptx}{\partial_{tx}^2}

\renewcommand{\i}{\ifmmode\mathit{\mathchar"7010 }\else\char"10 \fi}
\renewcommand{\j}{\ifmmode\mathit{\mathchar"7011 }\else\char"11 \fi}
\newcommand{\R}{\mathbb{R}}
\newcommand{\N}{\mathbb{N}}

\newcommand{\Hneg}{H_{\mathrm{loc}}^{-1}}
\newcommand{\CL}{\mathcal{L}}

\newcommand{\oeps}{\omega_{\eps}}

{%

\begin{enumerate}}%
{\end{enumerate}}

{%

\begin{enumerate}}%
{\end{enumerate}}

\begin{document}\large

\title[Dissipation and short wave dispersion]{Wellposedness of the Ostrovsky--Hunter Equation\\ under the combined  effects of dissipation\\ and short wave dispersion}

\date{\today}

\author[G. M. Coclite and L. di Ruvo]{Giuseppe Maria Coclite and Lorenzo di Ruvo}
\address[Giuseppe Maria Coclite and Lorenzo di Ruvo]
{\newline Department of Mathematics,   University of Bari, via E. Orabona 4, 70125 Bari,   Italy}
\email[]{giuseppemaria.coclite@uniba.it, lorenzo.diruvo@uniba.it}
\urladdr{http://www.dm.uniba.it/Members/coclitegm/}

\thanks{The authors are members of the Gruppo Nazionale per l'Analisi Matematica, la Probabilit\`a e le loro Applicazioni (GNAMPA) of the Istituto Nazionale di Alta Matematica (INdAM)}




\begin{abstract}
The Ostrovsky-Hunter equation provides a model for small-amplitude long waves in a rotating fluid of finite depth. 
It is a nonlinear evolution equation. In this paper we study the well-posedness for the Cauchy problem associated to this equation in presence of  some weak dissipation effects.
\end{abstract}

\maketitle

\section{Introduction}
\label{sec:intro}
Many physical problems (such as  non-linear shallow-water waves and wave motion in
plasmas) are described by the  following nonlinear evolution equation
\begin{equation}
\label{eq:Kdv}
\pt u + \px f(u) -\alpha\pxx u - \beta \pxxx u=0, \quad \alpha, \beta \in \R, \quad  f(u)=\frac{u^2}{2},
\end{equation}
which was derived by Korteweg-deVries (see \cite{KdV}).
\eqref{eq:Kdv} is also known as the Korteweg-de Vries-Burgers equation (see \cite{CG, FM, Su}), where $\alpha\pxx u$ is a viscous
dissipation term.
If \eqref{eq:Kdv} describes the evolution of non-linear shallow-water waves, then the function $u(t,x)$ is the amplitude of an appropriate linear long wave mode, with linear long wave speed $C_0$. However, when the effects of background
rotation through the Coriolis parameter $\kappa$ need to be taken into account, an extra term is needed, and \eqref{eq:Kdv} is replaced by
\begin{equation}
\label{eq:OHbeta}
\px(\pt u+\px f(u)-\alpha\pxx u -\beta \pxxx u)=\gamma u,
\end{equation}
where $\gamma=\frac{\kappa^2}{2C_0}$ (see \cite{dR,HT}).
If $\alpha=\beta=0$, then \eqref{eq:OHbeta} reads
\begin{equation}
\label{eq:OH}
\px(\pt u+\px f(u))=\gamma u.
\end{equation}
\eqref{eq:OH} is known under different names such as the reduced Ostrovsky equation \cite{GP, P, S}, the
Ostrovsky-Hunter equation \cite{B}, the short-wave equation \cite{H}, and the Vakhnenko equation
\cite{MPV, PV}. The well-posedness of \eqref{eq:OH} in class of discontinuous solutions has been proved in \cite{Cd, Cd1}.

If $\alpha=0$, \eqref{eq:OHbeta} reads
\begin{equation}
\label{eq:OHbeta1}
\px(\pt u+\px f(u)-\beta \pxxx u)=\gamma u,
\end{equation}
which is known as the Ostrovsky equation (see \cite{O}).
Mathematical properties of  \eqref{eq:OHbeta1} were studied
recently in many details, including the local and global well-posedness in energy space
\cite{GL, LM, LV, T}, stability of solitary waves \cite{LL, L, LV:JDE}, wave breaking \cite{LPS},  and convergence of solutions in the
limit of the Korteweg-deVries equation \cite{LL:07, LV:JDE}.

Let us assume, in \eqref{eq:OHbeta}, that $\alpha=1,\,\beta=0$. Therefore, we have
\begin{equation}
\label{eq:wd1}
\px(\pt u+\px f(u)-\pxx u)=\gamma u.
\end{equation}
\eqref{eq:wd1} describes the combined effects of dissipation and short waves dispersion, and is analogous to the \eqref{eq:Kdv} for dissipative long waves. It
can be deduced considering two asymptotic expansions of the shallow water equations, first with respect to the rotation frequency and then with respect to the amplitude of the waves (see \cite{dR, HT}).

We are interested in the initial value problem for \eqref{eq:wd1}, so we augment \eqref{eq:wd1} with the initial condition
\begin{equation}
\label{eq:init}
u(0,x)=u_0(x), \qquad x\in\R,
\end{equation}
on which we assume that
\begin{equation}
\label{eq:assinit}
u_0\in L^1(\R)\cap L^{\infty}(\R),\quad\int_{\R}u_{0}(x)dx=0.
\end{equation}
On the function
\begin{equation}
\label{eq:def-di-P01}
P_{0}(x)=\int_{-\infty}^{x} u_{0}(y)dy, \quad x\in\R,
\end{equation}
we assume that
\begin{equation}
\label{eq:L-2P01}
\begin{split}
\norm{P_0}^2_{L^2(\R)}&=\int_{\R}\left(\int_{-\infty}^{x}u_{0}(y)dy\right)^2dx <\infty,\\
\int_{\R}P_0(x)dx&= \int_{\R}\left(\int_{-\infty}^{x}u_{0}(y)dy\right)dx=0.
\end{split}
\end{equation}
The flux $f$ is assumed to be smooth, genuinely nonlinear, and subquadratic, namely:
\begin{equation}
\label{eq:assflux1}
f\in C^2(\R),\qquad |\{f''=0\}|=0,\qquad \vert f'(u) \vert \le C_{0} \vert u \vert, \quad u\in\R,
\end{equation}
for some a positive constant $C_0$.

Integrating \eqref{eq:wd1} on $(-\infty,x)$ we gain the integro-differential formulation of problem \eqref{eq:wd1}, and \eqref{eq:init} (see \cite{LV})
\begin{equation}
\label{eq:OHw-u}
\begin{cases}
\pt u+\px f(u) =\gamma \int^x_{-\infty} u(t,y) dy+\pxx u,&\qquad t>0, \ x\in\R,\\
u(0,x)=u_0(x), &\qquad x\in\R,
\end{cases}
\end{equation}
that is equivalent to
\begin{equation}
\label{eq:OHw}
\begin{cases}
\pt u+\px f(u)=\gamma P+\pxx u,&\qquad t>0, \ x\in\R ,\\
\px P=u,&\qquad t>0, \ x\in\R,\\
P(t,-\infty)=0,&\qquad t>0, \\
u(0,x)=u_0(x), &\qquad x\in\R.
\end{cases}
\end{equation}
The main result of this paper is the following theorem.
\begin{theorem}\label{th:wellp}
Let $T>0$. Assume \eqref{eq:assinit}, \eqref{eq:def-di-P01}, \eqref{eq:L-2P01} and \eqref{eq:assflux1}. Then there exists
a unique classical solution for the Cauchy problem of \eqref{eq:OHw-u}, or \eqref{eq:OHw}, $u$ such that
\begin{equation}
\label{eq:uePe}
\begin{split}
u&\in L^{\infty}((0,T)\times\R)\cap C((0,T);H^\ell(\R)),\quad \forall\ell \in\N, \\
P&\in L^{\infty}((0,T)\times\R)\cap L^{2}((0,T)\times\R),\\
\int_{\R} u&(t,x) dx=0, \quad t\ge 0.
\end{split}
\end{equation}
Moreover, if $u$ and $v$ are two solutions of \eqref{eq:OHw-u}, or \eqref{eq:OHw}, the following inequality holds
\begin{equation}
\label{eq:l2-stability}
\norm{u(t,\cdot)-v(t,\cdot)}_{L^2(\R)}\le e^{C(T)t}\norm{u_{0}-v_{0}}_{L^2(\R)},
\end{equation}
for some suitable $C(T)>0$, and every $0\le t\le T$.
\end{theorem}
The existence argument is based on passing to limit using a compensated compactness argument \cite{TartarI} in the parabolic-elliptic approximation of \eqref{eq:OHw}:
\begin{equation*}
\pt\ued +\px f(\ued) = \gamma \Ped +\pxx \ued, \quad -\ed\pxx\Ped+\px\Ped = \ued.
\end{equation*}
The paper is organized as follows. In Section \ref{sec:vv} we prove several a priori estimates on the parabolic-elliptic. Those play a key role in the proof of our main result, that is given in Section \ref{sec:w12}.

\section{Parabolic-elliptic approximation}\label{sec:vv}
Our existence argument is based on passing to the limit in a parabolic-elliptic approximation.
Fix $0<\delta <1$, and let $\ued=\ued (t,x)$ be the unique classical solution of the following mixed problem \cite{CHK:ParEll}:
\begin{equation}
\label{eq:OHepsw1}
\begin{cases}
\pt \ued+\px f(\ued)=\gamma\Ped +\pxx\ued,&\quad t>0,\ x\in\R,\\
-\delta\pxx\Ped+\px\Ped=\ued,&\quad t>0,\ x\in\R,\\
\ued(0,x)=u_{\delta,0}(x),&\quad x\in\R,
\end{cases}
\end{equation}
where $u_{\delta,0}$ is a $C^\infty$ approximation of $u_{0}$ such that
\begin{equation}
\label{eq:u0eps}
\begin{split}
&\norm{u_{\delta,0}}_{L^2(\R)}\le \norm{u_{0}}_{L^2(\R)},\quad \norm{u_{\delta,0}}_{L^{\infty}(\R)}\le \norm{u_{0}}_{L^{\infty}(\R)},\\
&\norm{\px u_{\delta,0}}_{L^2(\R)}\le C_{0},\quad   \norm{\pxx u_{\delta,0}}_{L^2(\R)}\le C_{0}\\
&\norm{P_{\delta,0}}_{L^2(\R)}\le \norm{P_{0}}_{L^2(\R)},\quad\delta\norm{\px P_{\delta,0}}_{L^2(\R)}\le C_{0},
\end{split}
\end{equation}
and $C_0$ is a constant independent on $\delta$.

Let us prove some a priori estimates on $\ued$ and $\Ped$, denoting with $C_0$ the constants which depend on the initial data, and $C(T)$ the constants which depend also on $T$.
\begin{lemma}
\label{lm:cns1}
For each $t\in (0,\infty)$,
\begin{equation}
\label{eq:P-pxP-intfy1}
\Ped(t,\infty)=\px\Ped(t,-\infty)=\px\Ped(t,\infty)=0.
\end{equation}
Moreover,
\begin{equation}
\label{eq:equ-L2-stima1}
\delta^2\norm{\pxx\Ped(t,\cdot)}^2_{L^2(\R)}+ \norm{\px\Ped(t,\cdot)}^2_{L^2(\R)}=\norm{\ued(t,\cdot)}^2_{L^2(\R)}.
\end{equation}
\end{lemma}
\begin{proof}
We begin by proving that \eqref{eq:P-pxP-intfy1} holds.

Differentiating the first equation of \eqref{eq:OHepsw1} with respect to $x$, we have
\begin{equation}
\label{eq:pxu}
\px(\pt \ued+\px f(\ued)-\pxx\ued)=\gamma\px\Ped.
\end{equation}
From the the smoothness of $\ued$, it follows from \eqref{eq:OHepsw1} and \eqref{eq:pxu} that
\begin{align*}
&\lim_{x\to\infty}(\pt \ued +\px f(\ued)-\pxx\ued)=\gamma\Ped(t,\infty)=0,\\
&\lim_{x\to -\infty}\px(\pt \ued+\px f(\ued)-\pxx\ued)=\gamma\px\Ped(t,-\infty)=0,\\
&\lim_{x\to\infty}\px(\pt \ued+ \px f(\ued)-\pxx\ued)=\gamma\px\Ped(t,\infty)=0,
\end{align*}
which gives \eqref{eq:P-pxP-intfy1}.

Let us show that \eqref{eq:equ-L2-stima1} holds.
Squaring the equation for $\Ped$ in \eqref{eq:OHepsw1}, we get
\begin{equation*}
\delta^2(\pxx\Ped)^2+(\px\Ped)^2 - \delta\px((\px\Ped)^2)=\ued^2.
\end{equation*}
Therefore, \eqref{eq:equ-L2-stima1} follows from  \eqref{eq:P-pxP-intfy1} and an integration on $\R$.
\end{proof}

\begin{lemma}
\label{lm:2}
For each $t\in(0,\infty)$,
\begin{align}
\label{eq:L-infty-Px}
\sqrt{\delta}\norm{\px\Ped(t, \cdot)}_{L^{\infty}(\R)}&\le \norm{\ued(t,\cdot)}_{L^2(\R)},\\
\label{eq:uP}
\int_{\R}\ued(t,x)\Ped(t,x) dx&\le \norm{\ued(t,\cdot)}^2_{L^2(\R)}.
\end{align}
\end{lemma}

\begin{proof}
We begin by proving that \eqref{eq:L-infty-Px} holds.\\
Observe that
\begin{equation*}
0\le (-\delta \pxx\Pe + \px\Pe)^2= \delta^2(\pxx\Pe)^2 +(\px\Pe)^2 - \delta\px((\px\Pe)^2),
\end{equation*}
that is,
\begin{equation}
\label{eq:equa-pxP}
\delta\px((\px\Ped)^2)\le \delta^2(\pxx\Ped)^2 +(\px\Ped)^2.
\end{equation}
Integrating \eqref{eq:equa-pxP} on $(-\infty, x)$, we have
\begin{equation}
\label{eq:equa-pxP1}
\begin{split}
\delta(\px\Ped)^2  &\le \delta^2\int_{-\infty}^{x}(\pxx\Ped)^2 dx +\int_{-\infty}^{x}(\px\Ped)^2 dx\\
&\le \delta^2\int_{\R}(\pxx\Ped)^2 dx +\int_{\R}(\px\Ped)^2 dx.
\end{split}
\end{equation}
It follows from \eqref{eq:equ-L2-stima1} and \eqref{eq:equa-pxP1} that
\begin{equation*}
\delta(\px\Ped)^2\le \delta^2\int_{\R}(\pxx\Ped)^2 dx +\int_{\R}(\px\Ped)^2 dx= \norm{\ued(t,\cdot)}^2_{L^2(\R)}.
\end{equation*}
Therefore,
\begin{equation*}
\sqrt{\delta}\vert \px\Ped(t,x)\vert \le \norm{\ued(t,\cdot)}_{L^2(\R)},
\end{equation*}
which gives \eqref{eq:L-infty-Px}.

Finally, we prove \eqref{eq:uP}.
Multiplying by $\Ped$ the equation for $\Ped$ in \eqref{eq:OHepsw1}, we get
\begin{equation*}
-\delta\Ped\pxx\Ped + \Ped\px\Ped= \ued\Ped.
\end{equation*}
An integration on $\R$ and \eqref{eq:P-pxP-intfy1} give
\begin{align*}
\int_{\R}\ued\Ped dx=&\frac{1}{2}\int_{\R}\px(\Pe)^2 dx - \delta \int_{\R}\Ped\pxx\Ped dx\\
=&-\delta\int_{\R}\Ped\pxx\Ped dx=\delta\int_{\R}(\px\Ped)^2 dx,
\end{align*}
that is
\begin{equation*}
\int_{\R}\ued\Ped dx = \delta\int_{\R}(\px\Ped)^2 dx.
\end{equation*}
Since $0<\delta <1$, from \eqref{eq:equ-L2-stima1}, we have \eqref{eq:uP}.
\end{proof}

\begin{lemma}
\label{lm:l2-u}
For each $t\in(0,\infty)$, the following inequality holds
\begin{equation}
\label{eq:l2-u}
\norm{\ued(t,\cdot)}_{L^2(\R)}^2+2e^{2\gamma t}
\int_0^t e^{-2\gamma s}\norm{\px \ued(s,\cdot)}^2_{L^2(\R)}ds\le e^{2\gamma t}\norm{u_{0}}^2_{L^2(\R)}.
\end{equation}
In particular, we have
\begin{equation}
\label{eq:h2-P}
\norm{\px \Ped(t,\cdot)}_{L^2(\R)},\,\delta \norm{\pxx \Ped(t,\cdot)}_{L^2(\R)},\,\sqrt{\delta}\norm{\px\Ped(t,\cdot)}_{L^\infty(\R)}\le  e^{\gamma t}\norm{u_{0}}_{L^2(\R)}.
\end{equation}
\end{lemma}

\begin{proof}
Due to \eqref{eq:OHepsw1} and \eqref{eq:uP},
\begin{align*}
\frac{d}{dt}\int_{\R} \ued^2dx=&2\int_{\R} \ued\pt\ued dx\\
=&2\int_{\R}\ued\pxx\ued dx-2\int_{\R} \ued f'(\ued)\px\ued dx+2\gamma\int_{\R} \ued\Ped dx\\
\le&-2\int_{\R}(\px\ued)^2 dx+2\gamma\norm{\ued(t,\cdot)}_{L^2(\R)}^2.
\end{align*}
The Gronwall Lemma and  \eqref{eq:u0eps} give \eqref{eq:l2-u}.

Finally, \eqref{eq:h2-P} follows from \eqref{eq:equ-L2-stima1}, \eqref{eq:L-infty-Px} and \eqref{eq:l2-u}.
\end{proof}

\begin{lemma}
\label{lm:p8}
For each $t\ge 0$, we have that
\begin{align}
\label{eq:intp-infty}
\int_{0}^{-\infty}\Ped(t,x)dx&=a_{\delta}(t), \\
\label{eq:int+infty}
\int_{0}^{\infty}\Ped(t,x)dx&=a_{\delta}(t),
\end{align}
where
\begin{equation}
a_{\delta}(t)=\frac{\delta}{\gamma}\ptx\Ped(t,0)- \frac{1}{\gamma}\pt \Ped(t,0)+ \frac{1}{\gamma}f(0)-\frac{1}{\gamma}f(\ued(t,0)) + \frac{1}{\gamma}\px\ued(t,0).
\end{equation}
In particular,
\begin{equation}
\label{eq:Pmedianulla}
\int_{\R}\Ped(t,x)dx=0, \quad t\geq 0.
\end{equation}
\end{lemma}
\begin{proof}
We begin by observing that, integrating the second equation of \eqref{eq:OHepsw1} on $(0,x)$, we have that
\begin{equation}
\label{eq:P-in-0}
\int_{0}^{x} \ued(t,y)dy = \Ped(t,x)-\Ped(t,0)-\delta\px\Ped(t,x)+\delta\px\Ped(t,0).
\end{equation}
It follows from \eqref{eq:P-pxP-intfy1} that
\begin{equation}
\label{eq:lim-int-u}
\lim_{x\to -\infty}\int_{0}^{x} \ued(t,y)dy=\int_{0}^{-\infty}\ued(t,x)dx =  \delta\px\Ped(t,0) -  \Ped(t,0).
\end{equation}
Differentiating \eqref{eq:lim-int-u} with respect to $t$, we get
\begin{equation}
\label{eq:lim-int-u-in-t}
\frac{d}{dt}\int_{0}^{-\infty}\ued(t,x)dx= \int_{0}^{-\infty}\pt\ued(t,x)dx=\delta\ptx\Ped(t,0) - \pt \Ped(t,0).
\end{equation}
Integrating the first equation of \eqref{eq:OHepsw1} on $(0,x)$, we obtain that
\begin{equation}
\begin{split}
\label{eq:int-1-eq}
\int_{0}^{x}\pt\ued(t,y) dy &+ f(\ued(t,x))-f(\ued(t,0))\\
&-\px\ued(t,x)+ \px\ued(t,0)=\gamma\int_{0}^{x}\Ped(t,y)dy.
\end{split}
\end{equation}
Being $\ued$ a smooth solution of \eqref{eq:OHepsw1}, we get
\begin{equation}
\label{eq:500}
\lim_{x\to-\infty}\Big( f(\ued(t,x))-\px\ued(t,x)\Big)=f(0).
\end{equation}
Sending $x\to -\infty$ in \eqref{eq:int-1-eq}, from \eqref{eq:lim-int-u-in-t} and \eqref{eq:500}, we have
\begin{align*}
\gamma\int_{0}^{-\infty}\Ped(t,x)dx= &\delta\ptx\Ped(t,0) - \pt \Ped(t,0)\\
&+f(0)-f(\ued(t,0)) +  \px\ued(t,0),
\end{align*}
which gives \eqref{eq:intp-infty}.

Let us show that \eqref{eq:int+infty} holds. We begin by observing that, for \eqref{eq:P-pxP-intfy1} and \eqref{eq:P-in-0},
\begin{equation*}
\int_{0}^{\infty}\ued(t,x)dx =  \delta\px\Ped(t,0) -  \Ped(t,0).
\end{equation*}
Therefore,
\begin{equation}
\label{eq:lim-int-u-1}
\lim_{x\to \infty}\int_{0}^{x}\pt \ued(t,y)dy=\int_{0}^{\infty}\pt\ued(t,x)dx =  \delta\ptx\Ped(t,0) -  \pt\Ped(t,0).
\end{equation}
Again by the regularity of $\ued$,
\begin{equation}
\label{eq:510}
\lim_{x\to\infty}\Big( f(\ued(t,x))-\px\ued(t,x)\Big)=f(0).
\end{equation}
It follows from \eqref{eq:int-1-eq}, \eqref{eq:lim-int-u-1} and \eqref{eq:510} that
\begin{align*}
\gamma\int_{0}^{\infty}\Ped(t,x)dx= &\delta\ptx\Ped(t,0) - \pt \Ped(t,0)\\
&+f(0)-f(\ued(t,0)) + \px\ued(t,0),
\end{align*}
which gives \eqref{eq:int+infty}.

Finally, we prove \eqref{eq:Pmedianulla}. It follows from \eqref{eq:intp-infty} that
\begin{equation*}
\int_{-\infty}^{0}\Ped(t,x)dx = -a_{\delta}(t).
\end{equation*}
Therefore, for \eqref{eq:int+infty},
\begin{align*}
\int_{-\infty}^{0}\Ped(t,x)dx+\int_{0}^{\infty}\Ped(t,x)=\int_{\R} \Ped(t,x)dx =-a_{\delta}(t)+a_{\delta}(t)=0,
\end{align*}
that is \eqref{eq:Pmedianulla}.
\end{proof}

Lemma \ref{lm:p8} says that $\Ped(t,x)$  is integrable at $\pm\infty$.
Therefore, for each $t\ge 0$, we can consider the following function
\begin{equation}
\label{eq:F1}
\Fed(t,x)=\int_{-\infty}^{x}\Ped(t,y)dy.
\end{equation}

\begin{lemma}
\label{lm:P-infty}
Let $T>0$. There exists  $C(T)>0$, independent on $\delta$, such that
\begin{align}
\label{eq:P-infty}
\norm{\Ped}_{L^{\infty}(I_{T,1})}&\le C(T),\\
\label{eq:l2P}
\norm{\Ped(t,\cdot)}_{L^2(\R)}&\le C(T),\\
\label{eq:l2pxP}
\delta\norm{\px\Ped(t,\cdot)}_{L^2(\R)}&\le C(T),
\end{align}
where
\begin{equation}
\label{eq:defI}
I_{T,1}=(0,T)\times\R.
\end{equation}
In particular, we have
\begin{equation}
\label{eq:pt-px-P}
\delta\left\vert\int_{0}^{t}\!\!\!\int_{\R}\Ped\ptx\Ped ds dx\right\vert\le C(T), \quad 0<t<T.
\end{equation}
\end{lemma}
\begin{proof}
Integrating the second equation of \eqref{eq:OHepsw1} on $(-\infty, x)$, for \eqref{eq:P-pxP-intfy1},  we have that
\begin{equation}
\label{eq:1550}
\int_{-\infty}^{x} \ued(t,y)dy=\Ped(t,x) -\delta\px\Ped(t,x).
\end{equation}
Differentiating \eqref{eq:1550} with respect to $t$, we get
\begin{equation}
\label{eq:1551}
\frac{d}{dt}\int_{-\infty}^{x} \ued(t,y)dy=\int_{-\infty}^{x}\pt \ued(t,y)dy=\pt\Ped(t,x) -\delta\ptx\Ped(t,x).
\end{equation}
It follows from an integration of the first equation of \eqref{eq:OHepsw1} on $(-\infty, x)$ and \eqref{eq:F1} that
\begin{equation}
\label{eq:1552}
\int_{-\infty}^{x}\pt\ued(t,y)dy + f(\ued(t,x)) - \px\ued(t,x)=\gamma\Fed(t,x).
\end{equation}
Due to \eqref{eq:1551} and \eqref{eq:1552}, we have
\begin{equation}
\label{eq:1554}
\pt\Ped(t,x)-\delta\ptx\Ped(t,x) =\gamma\Fed(t,x)- f(\ued(t,x)) +\px\ued(t,x).
\end{equation}
Multiplying \eqref{eq:1554} by $\Ped - \delta\px\Ped$, we have
\begin{equation}
\label{eq:1555}
\begin{split}
(\pt\Ped-\delta\ptx\Ped)(\Ped - \delta\px\Ped)= &\gamma\Fed(\Ped - \delta\px\Ped)\\
&- f(\ued)(\Ped - \delta\px\Ped)\\
&+\px\ued(\Ped - \delta\px\Ped).
\end{split}
\end{equation}
Integrating \eqref{eq:1555} on $(0,x)$, we have
\begin{equation}
\label{eq:1222}
\begin{split}
\int_{0}^{x}\pt\Ped\Ped dy&-\delta\int_{0}^{x} \pt\Ped\px\Ped dy\\
&-\delta \int_{0}^{x}\Ped\ptx\Ped dy +\delta^2\int_{0}^{x}\ptx\Ped\px\Ped dy\\
=& \gamma\int_{0}^{x}\Fed\Ped dy - \gamma\delta\int_{0}^{x} \Fed\px\Ped dy\\
&-\int_{0}^{x}f(\ued)\Ped dy + \delta \int_{0}^{x}f(\ued)\px\Ped dy\\
&+\int_{0}^{x}\px\ued\Ped dy - \delta\int_{0}^{x}\px\ued\px\Ped dy.
\end{split}
\end{equation}
We observe that
\begin{equation}
\label{eq:int-by-part}
-\delta \int_{0}^{x}\px\Ped\pt\Ped dy=-\delta\Ped\pt\Ped +\delta\Ped(t,0)\pt\Ped(t,0) + \delta\int_{0}^{x}\Ped\ptx\Ped dy.
\end{equation}
Therefore, \eqref{eq:1222} and \eqref{eq:int-by-part} give
\begin{equation}
\begin{split}
\label{eq:1223}
\int_{0}^{x}\pt\Ped\Ped dy&+ \delta^2\int_{0}^{x}\ptx\Ped\px\Ped dy\\
=& \delta\Ped\pt\Ped  -\delta\Ped(t,0)\pt\Ped(t,0)  + \gamma\int_{0}^{x}\Fed\Ped dy \\
&- \gamma\delta\int_{0}^{x} \Fed\px\Ped dy-\int_{0}^{x}f(\ued)\Ped dy + \delta \int_{0}^{x}f(\ued)\px\Ped dy\\
&+\int_{0}^{x}\px\ued\Ped dy - \delta\int_{0}^{x}\px\ued\px\Ped dy.
\end{split}
\end{equation}
Sending $x\to -\infty$, for \eqref{eq:P-pxP-intfy1}, we get
\begin{equation}
\label{eq:0012}
\begin{split}
\int_{0}^{-\infty}\pt\Ped\Ped dy&+ \delta^2\int_{0}^{-\infty}\ptx\Ped\px\Ped dy\\
=& -\delta\Ped(t,0)\pt\Ped(t,0) + \gamma\int_{0}^{-\infty}\Fed\Ped dy \\
&- \gamma\delta\int_{0}^{-\infty} \Fed\px\Ped dy-\int_{0}^{-\infty}f(\ued)\Ped dy \\
&+ \delta \int_{0}^{-\infty}f(\ued)\px\Ped dy+\int_{0}^{-\infty}\px\ued\Ped dy\\
& - \delta\int_{0}^{-\infty}\px\ued\px\Ped dy,
\end{split}
\end{equation}
while sending $x\to\infty$,
\begin{equation}
\label{eq:0013}
\begin{split}
\int_{0}^{\infty}\pt\Ped\Ped dy&+ \delta^2\int_{0}^{\infty}\ptx\Ped\px\Ped dy\\
=& -\delta\Ped(t,0)\pt\Ped(t,0) + \gamma\int_{0}^{\infty}\Fed\Ped dy- \gamma\delta\int_{0}^{\infty} \Fed\px\Ped dy \\
&-\int_{0}^{\infty}f(\ued)\Ped dy + \delta \int_{0}^{\infty}f(\ued)\px\Ped dy\\
&+\int_{0}^{\infty}\px\ued\Ped dy - \delta\int_{0}^{\infty}\px\ued\px\Ped dy.
\end{split}
\end{equation}
Since
\begin{align*}
\int_{\R}\Ped\pt\Ped dx &=\frac{1}{2}\frac{d}{dt}\int_{\R}\Ped^2dx,\\
\delta^2\int_{\R}\ptx\Ped\px\Ped dx &= \frac{\delta^2}{2}\frac{d}{dt}\int_{\R}(\px\Ped)^2dx,
\end{align*}
it follows from \eqref{eq:0012} and \eqref{eq:0013} that
\begin{equation}
\label{eq:12312}
\begin{split}
\frac{1}{2}\frac{d}{dt}\int_{\R}\Ped^2dx&+\frac{\delta^2}{2}\frac{d}{dt}\int_{\R}(\px\Ped)^2dx\\
=& \gamma\int_{\R}\Fed\Ped dx - \gamma\delta\int_{\R} \Fed\px\Ped dx\\
&-\int_{\R}f(\ued)\Ped dx + \delta \int_{\R}f(\ued)\px\Ped dx\\
&+\int_{\R}\px\ued\Ped dx - \delta\int_{\R}\px\ued\px\Ped dx.
\end{split}
\end{equation}
Due to \eqref{eq:Pmedianulla} and \eqref{eq:F1},
\begin{equation}
\label{eq:F-in-infty}
\begin{split}
2\gamma\int_{\R}\Fed\Ped dx&=2\gamma\int_{\R}\Fed\px\Fed dx =\gamma(\Fed(t,\infty))^2\\
&=\gamma\left( \int_{\R} \Ped(t,x)dx \right)^2=0.
\end{split}
\end{equation}
\eqref{eq:12312} and \eqref{eq:F-in-infty} give
\begin{equation}
\label{eq:12234}
\begin{split}
&\frac{d}{dt}\left(\int_{\R}\Ped^2dx + \delta^2\int_{\R}(\px\Ped)^2dx\right)\\
&\quad=-2\gamma\delta\int_{\R} \Fed\px\Ped dx -2\int_{\R}f(\ued)\Ped dx\\
&\qquad + 2\delta \int_{\R}f(\ued)\px\Ped dx +2\int_{\R}\px\ued\Ped dx\\
&\qquad - 2\delta\int_{\R}\px\ued\px\Ped dx.
\end{split}
\end{equation}
Thanks to \eqref{eq:P-pxP-intfy1}, \eqref{eq:Pmedianulla} and \eqref{eq:F1},
\begin{equation}
\label{eq:346}
-2\delta\gamma\int_{\R}\px\Ped\Fed dx=2\delta\gamma\int_{\R}\Ped\px\Fed dx = 2\delta\gamma\int_{\R} \Ped^2 dx\le 2\gamma \int_{\R} \Ped^2 dx,
\end{equation}
while for \eqref{eq:P-pxP-intfy1},
\begin{equation}
\label{eq:347}
\begin{split}
2\int_{\R}\px\ued\Ped dx=&-2\int_{\R}\ued\px\Ped dx.
\end{split}
\end{equation}
Hence, from \eqref{eq:assflux1}, \eqref{eq:346} and \eqref{eq:347},  we get
\begin{equation*}
\begin{split}
&\frac{d}{dt}\left(\int_{\R}\Ped^2dx + \delta^2\int_{\R}(\px\Ped)^2dx\right)\\
&\quad\le 2\gamma \int_{\R} \Ped^2 dx -2\int_{\R}f(\ued)\Ped dx + 2\delta \int_{\R}f(\ued)\px\Ped dx \\
&\qquad  -2  \int_{\R}\ued\px\Ped dx- 2\delta\int_{\R}\px\ued\px\Ped dx\\
&\quad\le 2\gamma \int_{\R} \Ped^2 dx+ 2\left\vert \int_{\R}f(\ued)\Ped dx\right\vert + 2\delta\left\vert\int_{\R}f(\ued)\px\Ped dx\right\vert\\
&\qquad   +2  \left\vert\int_{\R}\ued\px\Ped dx\right\vert+ 2\delta\left\vert\int_{\R}\px\ued\px\Ped dx\right\vert\\
&\quad\le 2\gamma \int_{\R} \Ped^2 dx+ 2\int_{\R}\vert f(\ued)\vert\vert\Ped \vert dx + 2\delta\int_{\R}\vert f(\ued)\vert \vert\px\Ped\vert dx\\
&\qquad  + 2\int_{\R}\vert\ued\vert\vert\px\Ped\vert dx + 2\delta\int_{\R}\vert\px\ued\vert\vert\px\Ped\vert dx\\
&\quad\le 2\gamma \int_{\R} \Ped^2 dx + 2C_{0}\int_{\R}\vert\Ped \vert\ued^2 dx + 2C_{0}\delta\int_{\R}\vert\px\Ped\vert\ued^2 dx\\
&\qquad + 2\int_{\R}\vert\ued\vert\vert\px\Ped\vert dx + 2\delta\int_{\R}\vert\px\ued\vert\vert\px\Ped\vert dx.
\end{split}
\end{equation*}
From the Young inequality,
\begin{align*}
2\int_{\R}\vert\px\Ped\vert\vert\ued\vert &\le\norm{\px\Ped(t,\cdot)}^2_{L^{2}(\R)}+  \norm{\ued(t,\cdot)}^2_{L^{2}(\R)},\\
2\delta\int_{\R}\vert\px\ued\vert\vert\px\Ped\vert dx&=\int_{\R}\left\vert\frac{\px\ued}{\sqrt{\gamma}}\right\vert\vert 2\sqrt{\gamma}\delta\px\Ped\vert dx\\
& \le \frac{1}{2\gamma}\norm{\px\ued(t,\cdot)}^2_{L^{2}(\R)} + 2\delta^2\gamma\norm{\px\Ped(t,\cdot)}^2_{L^{2}(\R)}.
\end{align*}
Thus,
\begin{equation}
\label{eq:350}
\begin{split}
\frac{d}{dt}G(t)-2\gamma G(t)\le & \norm{\ued(t,\cdot)}^2_{L^{2}(\R)}+ 2C_{0}\int_{\R}\vert\Ped \vert\ued^2 dx\\  &+2C_{0}\delta\int_{\R}\vert\px\Ped\vert\ued^2 dx + \norm{\px\Ped(t,\cdot)}^2_{L^{2}(\R)}\\
&+ \frac{1}{2\gamma}\norm{\px\ued(t,\cdot)}^2_{L^{2}(\R)},
\end{split}
\end{equation}
where
\begin{equation}
\label{eq:def-di-G}
G(t)=\norm{\Ped(t,\cdot)}^2_{L^2(\R)} + \delta^2\norm{\px\Ped(t,\cdot)}^2_{L^2(\R)}.
\end{equation}
We observe that, from \eqref{eq:l2-u},
\begin{equation}
\label{eq:399}
2C_{0}\int_{\R} \vert\Ped\vert \ued^2 dx\le C_{0}e^{2\gamma t} \norm{\Ped}_{L^{\infty}(I_{T,1})},
\end{equation}
where $I_{T,1}$ is defined in \eqref{eq:defI}.
Since $0<\delta <1$, it follows from \eqref{eq:l2-u} and \eqref{eq:h2-P} that
\begin{equation}
\label{eq:400}
\begin{split}
2C_{0}\delta\int_{\R}\vert\px\Ped\vert \ued^2dx&\le 2C_{0}\delta\norm{\px\Ped(t,\cdot)}_{L^{\infty}(\R)}\norm{\ued(t,\cdot)}^2_{L^2(\R)}\\
&\le 2\sqrt{\delta}C_{0}e^{3\gamma t}\le C_{0}e^{3\gamma t}.
\end{split}
\end{equation}
Again by \eqref{eq:h2-P}, we have that
\begin{equation}
\label{eq:401}
\norm{\px\Ped(t,\cdot)}^2_{L^{2}(\R)}\le C_{0}e^{2\gamma t}.
\end{equation}
Therefore, \eqref{eq:l2-u}, \eqref{eq:400} and \eqref{eq:401} give
\begin{equation*}
\frac{d}{dt}G(t)-2\gamma G(t)\le  C_{0}\left( \norm{\Ped}_{L^{\infty}(I_{T,1})}+ 1 \right)e^{2\gamma t} + C_{0}e^{3\gamma t}+ \frac{1}{2\gamma}\norm{\px\ued(t,\cdot)}^2_{L^{2}(\R)}.
\end{equation*}
The Gronwall Lemma, \eqref{eq:u0eps},  \eqref{eq:l2-u} and \eqref{eq:def-di-G} give
\begin{align*}
&\norm{\Ped(t,\cdot)}^2_{L^2(\R)} + \delta^2\norm{\px\Ped(t,\cdot)}^2_{L^2(\R)}\\
&\quad\le \norm{P_{0}}^2_{L^2(0,\infty)}e^{2\gamma t} +  \left( \norm{\Ped}_{L^{\infty}(I_{T,1})}+ 1 \right)t e^{2\gamma t} + C_{0}t e^{3\gamma t}\\
&\qquad  +\frac{ e^{2\gamma t}}{2\gamma}\int_{0}^{t} e^{-2\gamma s} \norm{\px\ued(s,\cdot)}^2_{L^{2}(\R)}ds\\
&\quad \le  \norm{P_{0}}^2_{L^2(0,\infty)}e^{2\gamma t} +  \left( \norm{\Ped}_{L^{\infty}(I_{T,1})}+ 1 \right)t e^{2\gamma t} + C_{0}t e^{3\gamma t}+C_{0}e^{2\gamma t}.
\end{align*}
Hence,
\begin{equation}
\label{eq:505}
\norm{\Ped(t,\cdot)}^2_{L^2(\R)} + \delta^2\norm{\px\Ped(t,\cdot)}^2_{L^2(\R)}\le C(T)\left( \norm{\Ped}_{L^{\infty}(I_{T,1})}+ 1 \right).
\end{equation}
Due to \eqref{eq:h2-P}, \eqref{eq:505} and the H\"older inequality,
\begin{align*}
\Ped^2(t,x)&\le 2\int_{\R}\vert\Ped\vert\vert\px\Ped\vert dx \le 2\norm{\Ped(t,\cdot)}_{L^2(\R)} \norm{\px\Ped(t,\cdot)}_{L^2(\R)}\\
&\le 2\sqrt{C(T)\left(\norm{\Ped}_{L^{\infty}(I_{T,1})}+1\right)}\sqrt{C_{0}}e^{\gamma t}\le C(T)\left(\norm{\Ped}_{L^{\infty}(I_{T,1})}+1\right).
\end{align*}
Therefore,
\begin{equation*}
\norm{\Ped}^2_{L^{\infty}(I_{T,1})} - C(T) \norm{\Ped}_{L^{\infty}(I_{T,1})} - C(T)\le 0,
\end{equation*}
which gives \eqref{eq:P-infty}.

\eqref{eq:l2P} and \eqref{eq:l2pxP} follow from \eqref{eq:P-infty} and \eqref{eq:505}.

Let us show that \eqref{eq:pt-px-P} holds. Multiplying \eqref{eq:1554} by $\Ped$, an integration on $\R$ and \eqref{eq:F-in-infty} give
\begin{align*}
2\delta\int_{\R}\ptx\Ped\Ped dx =& \frac{d}{dt} \norm{\Ped(t,\cdot)}^2_{L^2(\R)} -2\gamma\int_{\R}\Fed\Ped dx\\
&+2\int_{\R} f(\ued)\Ped dx -2\int_{\R}\px\ued\Ped dx\\
=&\frac{d}{dt} \norm{\Ped(t,\cdot)}^2_{L^2(\R)}+2\int_{\R} f(\ued)\Ped dx -2\int_{\R}\px\ued\Ped dx.
\end{align*}
An integration on $(0,t)$ gives
\begin{align*}
2\delta\int_{0}^{t}\!\!\!\int_{\R}\ptx\Ped\Ped dx =&\norm{\Ped(t,\cdot)}^2_{L^2(\R)}-\norm{P_{\eps,\delta,0}}^2_{L^2(\R)}\\
&+2\int_{0}^{t}\!\!\!\int_{\R}f(\ued)\Ped dx -2 \int_{0}^{t}\!\!\!\int_{\R} \px\ued\Ped dx.
\end{align*}
It follows from \eqref{eq:assflux1}, \eqref{eq:l2-u}, \eqref{eq:P-infty} and \eqref{eq:l2P} that
\begin{align*}
2\delta\left\vert\int_{0}^{t}\!\!\!\int_{\R}\ptx\Ped\Ped dsdx\right\vert \le&\norm{\Ped(t,\cdot)}^2_{L^2(\R)}+\norm{P_{\eps,\delta,0}}^2_{L^2(\R)}\\
&+2\int_{0}^{t}\!\!\!\int_{\R}\vert f(\ued)\vert\vert\Ped\vert dsdx\\
& +2 \int_{0}^{t}\!\!\!\int_{\R} \vert\px\ued\vert\vert\Ped\vert dsdx\\
\le & \norm{P_{\delta,0}}^2_{L^2(\R)}+ 2C(T)\int_{0}^{t}\!\!\!\int_{\R}\ued^2 dsdx\\
 &+2 \int_{0}^{t}\!\!\!\int_{\R} \vert\px\ued\vert\vert\Ped\vert dsdx +C(T)\\
\le & \norm{P_{\delta,0}}^2_{L^2(\R)}  +C(T)\\
&+2 \int_{0}^{t}\!\!\!\int_{\R} \vert\px\ued\vert\vert\Ped\vert dsdx.
\end{align*}
Observe that, thanks to \eqref{eq:l2-u},
\begin{equation}
\label{eq:pxU2}
\begin{split}
&\int_{0}^{t}\norm{\px\ued(s,\cdot)}^2_{L^2(\R)}ds\\
&\quad \le  e^{2\gamma t}\int_{0}^{t}e^{-2\gamma s}\norm{\px\ued(s,\cdot)}^2_{L^2(\R)}ds\le C(T).
\end{split}
\end{equation}
Due to the Young inequality,
\begin{equation}
\label{eq:young}
\begin{split}
&2\int_{\R} \vert\px\ued\vert\vert\Ped\vert dsdx\\
&\quad\le \norm{\Ped(t,\cdot)}^2_{L^2(\R)}+\norm{\px\ued(t,\cdot)}^2_{L^2(\R)}\\
& \quad \le C(T) +\norm{\px\ued(t,\cdot)}^2_{L^2(\R)}.
\end{split}
\end{equation}
Then, from \eqref{eq:pxU2} and  \eqref{eq:young}, we have that
\begin{align*}
&2\int_{0}^{t}\!\!\!\int_{\R}\vert\Ped\vert\vert\px\ued\vert dsdx\\
&\quad \le \int_{0}^{t}\norm{\Ped(s,\cdot)}^2_{L^2(\R)} ds +  \int_{0}^{t}\norm{\px\ued(s,\cdot)}^2_{L^2(\R)} ds\le C(T)  .
\end{align*}
Therefore,
\begin{equation*}
2\delta\left\vert\int_{0}^{t}\!\!\!\int_{\R}\Ped\ptx\Ped dsdx\right\vert\le \norm{P_{\eps,0}}^2_{L^2(\R)}+C(T),
\end{equation*}
which gives \eqref{eq:pt-px-P}.
\end{proof}
\begin{lemma}
\label{lm:linfty-u}
Let $T>0$. Then,
\begin{equation}
\label{eq:linfty-u}
\norm{\ued}_{L^\infty(I_{T,1})}\le \norm{u_{0}}_{L^\infty(\R)}+ C(T),
\end{equation}
where $I_{T,1}$ is defined in \eqref{eq:defI}.
\end{lemma}
\begin{proof}
Due to \eqref{eq:OHepsw1} and \eqref{eq:P-infty},
\begin{equation*}
\pt \ued +\px f(\ued)-\pxx \ued\le \gamma C(T).
\end{equation*}
Since the map
\begin{equation*}
{\mathcal F}(t):=\norm{u_{0}}_{L^\infty(\R)}+\gamma C(T)t,
\end{equation*}
solves the equation
\begin{equation*}
\frac{d{\mathcal F}}{dt}=\gamma C(T)
\end{equation*}
and
\begin{equation*}
\max\{\ued(0,x),0\}\le {\mathcal F}(t),\qquad (t,x)\in I_{T,1},
\end{equation*}
the comparison principle for parabolic equations implies that
\begin{equation*}
 \ued(t,x)\le {\mathcal F}(t),\qquad (t,x)\in I_{T,1}.
\end{equation*}
In a similar way we can prove that
\begin{equation*}
\ued(t,x)\ge -{\mathcal F}(t),\qquad (t,x)\in I_{T,1}.
\end{equation*}
Therefore,
\begin{equation*}
\vert\ued(t,x)\vert\le\norm{u_{0}}_{L^\infty(\R)}+\gamma C(T)t\le\norm{u_{0}}_{L^\infty(\R)}+ C(T),
\end{equation*}
which gives \eqref{eq:linfty-u}.
\end{proof}

\begin{lemma}\label{lm:34}
Let $T>0$ and $0<\delta<1$. We have that
\begin{equation}
\label{eq:012}
\norm{\px\ued(t,\cdot)}^2_{L^2(\R)} + \int_{0}^{t}\norm{\pxx\ued(s,\cdot)}^2_{L^2(\R)}ds\le C(T).
\end{equation}
\end{lemma}
\begin{proof}
Let $0<t<T$. Multiplying \eqref{eq:OHepsw1} by $-\pxx\ued$, we have
\begin{equation}
\label{eq:12346}
\begin{split}
-\pxx\ued\pt\ued &+(\pxx\ued)^2\\
 =&-\gamma\Ped\pxx\ued-  f'(\ued)\px\ued\pxx\ued.
\end{split}
\end{equation}
Since
\begin{equation*}
-\int_{\R}\pxx\ued\pt\ued dx=\frac{d}{dt}\left(\frac{1}{2}\int_{\R}(\px\ued)^2 \right),
\end{equation*}
integrating \eqref{eq:12346} on $\R$, we get
\begin{align*}
\frac{d}{dt}\left(\int_{\R}(\px\ued)^2 dx\right)&+2\int_{\R}(\pxx\ued)^2 dx\\
=& -2\gamma\int_{\R}\Ped\pxx\ued dx \\
&- 2\int_{\R}f'(\ued)\px\ued\pxx\ued dx.
\end{align*}
Due to \eqref{eq:l2-u}, \eqref{eq:l2P}, \eqref{eq:linfty-u} and the Young inequality,
\begin{align*}
&-2\gamma\int_{\R}\Ped\pxx\ued dx\\
&\qquad \le 2\gamma\left\vert\int_{\R}\Ped\pxx\ued dx\right\vert\\
&\qquad \le 2\int_{\R}\left\vert\sqrt{2}\gamma\Ped\right\vert\left\vert\frac{\pxx\ued}{\sqrt{2}}\right\vert dx\\
&\qquad \le 2\gamma^2\norm{\Ped(t,\cdot)}^2_{L^2(\R)} + \frac{1}{2} \norm{\pxx\ued(t,\cdot)}^2_{L^2(\R)}\\
&\qquad \le C(T)+ \frac{1}{2}  \norm{\pxx\ued(t,\cdot)}^2_{L^2(\R)},\\
&-2\int_{\R}f'(\ued)\px\ued\pxx\ued dx\\
&\qquad \le 2\left\vert\int_{\R}f'(\ued)\px\ued\pxx\ued dx\right\vert\\
&\qquad \le 2\int_{\R}\left\vert\sqrt{2}f'(\ued)\px\ued\right\vert \left\vert\frac{  \pxx\ued}{\sqrt{2}}\right \vert dx\\
&\qquad \le 2\int_{\R}(f'(\ued))^2(\px\ued^2)+ \frac{1}{2}\int_{\R}(\pxx\ued)^2 dx\\
&\qquad \le 2\norm{f'}^2_{L^{\infty}(I_{T,2})}\norm{ \px\ued(t,\cdot)}^2_{L^2(\R)}  +  \frac{1}{2}\norm{\pxx\ued(t,\cdot)}^2_{L^2(\R)},
\end{align*}
where
\begin{equation}
\label{eq:def-I2}
I_{T,2}=\left(-\norm{u_{0}}_{L^\infty(\R)}-C(T), \norm{u_{0}}_{L^\infty(\R)}+C(T)\right).
\end{equation}
Therefore,
\begin{align*}
\frac{d}{dt}\left(\norm{\px\ued(t,\cdot)}^2_{L^2(\R)} \right)&+2\norm{(\pxx\ued(t,\cdot))}^2_{L^2(\R)}\\
\le &   \norm{\pxx\ued(t,\cdot)}^2_{L^2(\R)}+ \norm{f'}^2_{L^{\infty}(I_{T,2})}\norm{ \px\ued(t,\cdot)}^2_{L^2(\R)} + C(T),  \\
\end{align*}
that is
\begin{align*}
\frac{d}{dt}\left(\norm{\px\ued(t,\cdot)}^2_{L^2(\R)} \right)&+ \norm{\pxx\ued(t,\cdot)}^2_{L^2(\R)}\\
\le &  \norm{f'}^2_{L^{\infty}(I_{T,2})}\norm{ \px\ued(t,\cdot)}^2_{L^2(\R)}+ C(T).
\end{align*}
An integration on $(0,t)$ and \eqref{eq:u0eps} give
\begin{equation}
\label{eq:001}
\begin{split}
\norm{\px\ued(t,\cdot)}^2_{L^2(\R)} &+ \int_{0}^{t}\norm{\pxx\ued(s,\cdot)}^2_{L^2(\R)}ds\\
\le & 2\norm{f'}^2_{L^{\infty}(I_{T,2})}\int_{0}^{t}\norm{ \px\ued(s,\cdot)}^2_{L^2(\R)}ds +C(T).
\end{split}
\end{equation}
\eqref{eq:012} follows from \eqref{eq:pxU2} and \eqref{eq:001}.
\end{proof}

\begin{lemma}\label{lm:37}
Let $T>0$ and $0<\delta<1$. We have that
\begin{equation}
\label{eq:052}
\norm{\px\ued}_{L^{\infty}(I_{T,1})} \le C(T),
\end{equation}
where $I_{T,1}$ is defined in \eqref{eq:defI}. Moreover,
\begin{equation}
\label{eq:055}
\norm{\pxx\ued(t,\cdot)}^2_{L^2(\R)}+\int_{0}^{t}\norm{\pxxx\ued(s,\cdot)}^2_{L^2(\R)}ds\le C(T).
\end{equation}

\end{lemma}
\begin{proof}
Let $0<t<T$. Multiplying \eqref{eq:OHepsw1} by $\pxxxx\ued$, we have
\begin{equation}
\label{eq:018}
\begin{split}
\pxxxx\ued\pt\ued &- \pxxxx\ued\pxx\ued\\
=&\gamma\Ped\pxxxx\ued- f'(\ued)\px\ued\pxxxx\ued.
\end{split}
\end{equation}
Since
\begin{align*}
\int_{\R}\pxxxx\ued\pt\ued dx =&\frac{d}{dt}\left(\frac{1}{2}\int_{\R}(\pxx\ued)^2 dx\right), \\
- \int_{\R}\pxxxx\ued\pxx\ued dx =& \int_{\R}(\pxxx\ued)^2 dx,\\
\gamma\int_{\R}\Ped\pxxxx\ued dx = &-\gamma\int_{\R}\px\Ped\pxxx\ued dx,\\
-\int_{\R}f'(\ued)\px\ued\pxxxx\ued dx=&\int_{\R}f''(\ued)(\px\ued)^2\pxxx\ued dx\\
& + \int_{\R}f'(\ued)\pxx\ued\pxxx\ued dx,
\end{align*}
integrating \eqref{eq:12346} on $\R$, we get
\begin{align*}
\frac{d}{dt}\left(\int_{\R}(\pxx\ued)^2 dx\right)&+2\int_{\R}(\pxxx\ued)^2 dx\\
=& -2\gamma\int_{\R}\px\Ped\pxxx\ued dx\\
&+2\int_{\R}f''(\ued)(\px\ued)^2\pxxx\ued dx\\
&+ 2\int_{\R}f'(\ued)\pxx\ued\pxxx\ued dx.
\end{align*}
Due to \eqref{eq:h2-P}, \eqref{eq:linfty-u}, \eqref{eq:012} and the Young inequality,
\begin{align*}
&-2\gamma\int_{\R}\px\Ped\pxxx\ued dx\\
&\qquad \le 2 \gamma\left\vert \int_{\R}\px\Ped\pxxx\ued dx \right\vert\\
&\qquad \le 2\int_{\R}\left\vert\sqrt{3}\gamma\px\Ped\right\vert\left\vert\frac{\pxxx\ued}{\sqrt{3}}\right\vert dx\\
&\qquad \le 3\gamma^2\norm{\px\Ped(t.\cdot)}^2_{L^{2}(\R)}+ \frac{1}{3}\norm{\pxxx\ued(t.\cdot)}^2_{L^{2}(\R)}\\
&\qquad \le C(T) + \frac{1}{3}\norm{\pxxx\ued(t.\cdot)}^2_{L^{2}(\R)},\\
&2\int_{\R}f''(\ued)(\px\ued)^2\pxxx\ued dx\\
&\qquad \le 2\left\vert\int_{\R}f''(\ued)(\px\ued)^2\pxxx\ued dx\right\vert\\
&\qquad \le 2\int_{\R}\left\vert\sqrt{3} f''(\ued)(\px\ued)^2\right\vert\left\vert\frac{\pxxx\ued}{\sqrt{3}}\right \vert dx\\
&\qquad \le 3\int_{\R}(f''(\ued))^2(\px\ued)^4 dx + \frac{1}{3}\norm{\pxxx\ued(t,\cdot)}^2_{L^2(\R)}\\
&\qquad \le 3\norm{f''}^2_{L^{\infty}(I_{T,2})}\norm{\px\ued}^2_{L^{\infty}(I_{T,1})}\norm{\px\ued(t,\cdot)}^2_{L^2(\R)}+ \frac{1}{3} \norm{\pxxx\ued(t,\cdot)}^2_{L^2(\R)}\\
&\qquad \le 3\norm{f''}^2_{L^{\infty}(I_{T,2})}C(T)\norm{\px\ued}^2_{L^{\infty}(I_{T,1})}\\
&\qquad\quad+ \frac{1}{3} \norm{\pxxx\ued(t,\cdot)}^2_{L^2(\R)},\\
& 2\int_{\R}f'(\ued)\pxx\ued\pxxx\ued dx\\
&\qquad \le 2\left\vert\int_{\R}f'(\ued)\pxx\ued\pxxx\ued dx\right\vert\\
&\qquad \le 2\int_{\R} \left\vert\sqrt{3}f'(\ued)\pxx\ued\right\vert\left \vert \frac{\pxxx\ued}{\sqrt{3}}\right \vert dx\\
&\qquad \le 3\int_{\R}(f'(\ued))^2(\pxx\ued)^2 dx + \frac{1}{3} \norm{\pxxx\ued(t,\cdot)}^2_{L^2(\R)}\\
&\qquad \le 3\norm{f'}^2_{L^{\infty}(I_{T,2})}\norm{\pxx\ued(t,\cdot)}^2_{L^2(\R)}+ \frac{1}{3}\norm{\pxxx\ued(t,\cdot)}^2_{L^2(\R)},
\end{align*}
where $I_{T,1}$ is defined in \eqref{eq:defI} and $I_{T,2}$ is defined in \eqref{eq:def-I2}. Therefore,
\begin{align*}
\frac{d}{dt}\left(\norm{\pxx\ued(t,\cdot)}^2_{L^2(\R)}\right)&+2\norm{\pxxx\ued(t,\cdot)}^2_{L^2(\R)}\\
\le &\norm{\pxxx\ued(t.\cdot)}^2_{L^{2}(\R)}\\
&+3\norm{f''}^2_{L^{\infty}(I_{T,2})}C(T)\norm{\px\ued}^2_{L^{\infty}(I_{T,1})}\\
&+ 3\norm{f'}^2_{L^{\infty}(I_{T,2})}\norm{\pxx\ued(t,\cdot)}^2_{L^2(\R)} +C(T),
\end{align*}
that is
\begin{align*}
\frac{d}{dt}\left(\norm{\pxx\ued(t,\cdot)}^2_{L^2(\R)}\right)&+\norm{\pxxx\ued(t,\cdot)}^2_{L^2(\R)}\\
\le & C(T)\norm{\px\ued}^2_{L^{\infty}(I_{T,1})} +C(T)\\
& + C(T)\norm{\pxx\ued(t,\cdot)}^2_{L^2(\R)}.
\end{align*}
An integration on $(0,t)$, \eqref{eq:u0eps} and \eqref{eq:012} give
\begin{align*}
\norm{\pxx\ued(t,\cdot)}^2_{L^2(\R)}&+\int_{0}^{t}\norm{\pxxx\ued(s,\cdot)}^2_{L^2(\R)}ds\\
\le &\left(C(T)\norm{\px\ued}^2_{L^{\infty}(I_{T,1})} +C(T)\right)\int_{0}^{t}ds\\
& + C(T)\int_{0}^{t}\norm{\pxx\ued(s,\cdot)}^2_{L^2(\R)}ds\\
\le & C(T)\norm{\px\ued}^2_{L^{\infty}(I_{T,1})} +C(T).
\end{align*}
Thus,
\begin{equation}
\label{eq:045}
\begin{split}
\norm{\pxx\ued(t,\cdot)}^2_{L^2(\R)}&+\int_{0}^{t}\norm{\pxxx\ued(s,\cdot)}^2_{L^2(\R)}ds\\
 \le& C(T)\left(1+\norm{\px\ued}^2_{L^{\infty}(I_{T,1})}\right).
\end{split}
\end{equation}
Due to \eqref{eq:012}, \eqref{eq:045} and the H\"older inequality,
\begin{align*}
(\px\ued(t,x))^2\le& 2\int_{\R}\vert\px\ued\vert\vert\pxx\ued\vert dx\\
\le &  2\norm{\px\ued(t,\cdot)}_{L^2(\R)} \norm{\pxx\ued(t,\cdot)}_{L^2(\R)}\\
\le & C(T)\sqrt{\left(1+\norm{\px\ued}^2_{L^{\infty}(I_{T,1})}\right)}.
\end{align*}
Then,
\begin{equation*}
\norm{\px\ued}^4_{L^{\infty}(I_{T,1})} -C(T)\norm{\px\ued}^2_{L^{\infty}(I_{T,1})} -C(T) \le 0,
\end{equation*}
which gives \eqref{eq:052}. \\
\eqref{eq:055} follows from \eqref{eq:052} and \eqref{eq:045}.
\end{proof}
Arguing as in \cite{CHK:ParEll}, we obtain the following result
\begin{lemma}\label{lm:38}
Let $T>0$, $\ell >2$  and $0<\delta <1$. For each $t\in (0,T)$,
\begin{equation}
\px^{\ell}\ued(t,\cdot)\in L^2(\R).
\end{equation}
\end{lemma}

\section{Proof of Theorem \ref{th:wellp}}\label{sec:w12}
This section is devoted to the proof of Theorem \ref{th:wellp}.

We begin by proving the following result
\begin{lemma}\label{lm:exist}
Let $T>0$. Assume  \eqref{eq:assinit}, \eqref{eq:def-di-P01}, \eqref{eq:L-2P01} and \eqref{eq:assflux1}. Then there exist
\begin{align}
\label{eq:uePe1}
u&\in L^{\infty}((0,T)\times\R)\cap C((0,T);H^\ell(\R)),\quad \ell>2,\\
\label{eq:uePe2}
P&\in L^{\infty}((0,T)\times\R)\cap L^{2}((0,T)\times\R),
\end{align}
where $u$ is a classical solution of the Cauchy problem of \eqref{eq:OHw}.
\end{lemma}

\begin{proof}
Let $\eta:\R\to\R$ be any convex $C^2$ entropy function, and
$q:\R\to\R$ be the corresponding entropy
flux defined by $q'=f'\eta'$.
By multiplying the first equation in \eqref{eq:OHepsw1} with
$\eta'(u)$ and using the chain rule, we get
\begin{equation*}
    \pt  \eta(\ued)+\px q(\ued)
    =\underbrace{ \pxx \eta(\ued)}_{=:\CL_{1,\delta}}
    \, \underbrace{- \eta''(\ued)\left(\px  \ued\right)^2}_{=: \CL_{2,\delta}}
     \, \underbrace{+\gamma\eta'(\ued) \Ped}_{=: \CL_{3,\delta}},
\end{equation*}
where  $\CL_{1,\delta}$, $\CL_{2,\delta}$, $\CL_{3,\delta}$ are distributions.

Let us show that
\begin{equation}
\label{eq:compct-H-1}
\textrm{$\{\CL_{1,\delta}\}_\delta$ is compact in $H^{-1}((0,T)\times\R)$, $T>0$}.
\end{equation}
Since
\begin{equation*}
\pxx\eta(\ued)=\px(\eta'(\ued)\px\ued),
\end{equation*}
we have to prove that
\begin{align}
\label{eq:eta12}
&\textrm{$\{\eta'(\ued)\px\ued\}_\delta$ is bounded in $L^2((0,T)\times\R)$, $T>0$},\\
\label{eq:eta13}
&\textrm{$\{\eta''(\ued)(\px\ued)^2+ \eta'(\ued)\pxx\ued\}_\delta$ is bounded in $L^2((0,T)\times\R)$, $T>0$}.
\end{align}
We begin by proving that \eqref{eq:eta12} holds.
Thanks to Lemmas \ref{lm:l2-u} and \ref{lm:linfty-u},
\begin{align*}
\norm{\eta'(\ued)\px\ued}^2_{L^2((0,T)\times\R)}&\le\norm{\eta'}^2_{L^{\infty}(I_{T,2})}\int_{0}^{T}\norm{\px\ued(s,\cdot)}^2_{L^2(\R)}ds\\
&\le \norm{\eta'}^2_{L^{\infty}(I_{T,2})} e^{2\gamma T}\int_{0}^{T}e^{-2\gamma s}\norm{\px\ued(s,\cdot)}^2_{L^2(\R)}ds\\
&\le \frac{1}{2}\norm{\eta'}^2_{L^{\infty}(I_{T,2})}e^{2\gamma T}\norm{u_{0}}^2_{L^2(\R)}\le C(T),
\end{align*}
where $I_{T,2}$ is defined in \eqref{eq:def-I2}.

We claim that
\begin{equation}
\label{eq:eta14}
\textrm{$\{\eta''(\ued)(\px\ued)^2\}_\delta$ is bounded in $L^2((0,T)\times\R)$}.
\end{equation}
Due to Lemmas \ref{lm:l2-u}, \ref{lm:linfty-u}, \ref{lm:37}
\begin{align*}
\norm{\eta''(\ued)(\px\ued)^2}^2_{L^2((0,T)\times\R)}&\le\norm{\eta''}^2_{L^{\infty}(I_{T,2})}\int_{0}^{T}\!\!\!\!\int_{\R}(\px\ued(s,x))^4dsdx\\
&\le \norm{\eta''}^2_{L^{\infty}(I_{T,2})}\norm{\px\ued}^2_{L^{\infty}(I_{T,1})}\int_{0}^{T}\norm{\px\ued(s,\cdot)}^2_{L^2(\R)}ds\\
&\le \frac{1}{2}\norm{\eta''}^2_{L^{\infty}(I_{T,2})}\norm{\px\ued}^2_{L^{\infty}(I_{T,1})}e^{2\gamma T}\norm{u_{0}}^2_{L^2(\R)}\le C(T),
\end{align*}
where $I_{T,1}$ is defined in \eqref{eq:defI}.

We claim that
\begin{equation}
\label{eq:eta15}
\textrm{$\{\eta'(\ued)\pxx\ued\}_\delta$ is bounded in $L^2((0,T)\times\R)$}.
\end{equation}
Thanks to Lemmas \ref{lm:linfty-u} and \ref{lm:34},
\begin{align*}
\norm{\eta'(\ued)\pxx\ued}^2_{L^2((0,T)\times\R)}&\le\norm{\eta'}^2_{L^{\infty}(I_{T,2})}\int_{0}^{T}\norm{\pxx\ued(s,\cdot)}^2_{L^2(\R)}ds\\
&\le \norm{\eta'}^2_{L^{\infty}(I_{T,2})}C(T)\le C(T).
\end{align*}
\eqref{eq:eta14} and \eqref{eq:eta15} give \eqref{eq:eta13}.\\
Therefore, \eqref{eq:compct-H-1} follows from \eqref{eq:eta12} and \eqref{eq:eta13}.

We have that
\begin{equation*}
\textrm{$\{\CL_{2,\delta}\}_{\delta>0}$ is bounded in $L^1((0,T)\times\R)$}.
\end{equation*}
Due to Lemmas \ref{lm:l2-u}, \ref{lm:linfty-u},
\begin{align*}
\norm{\eta''(\ued)(\px\ued)^2}_{L^1((0,T)\times\R)}&\le\norm{\eta''}_{L^{\infty}(I_{T,2})}\int_{0}^{T}\norm{\px\ued(s,\cdot)}^2_{L^2(\R)}ds\\
&\le \norm{\eta'}^2_{L^{\infty}(I_{T,2})} e^{2\gamma T}\int_{0}^{T}e^{-2\gamma s}\norm{\px\ued(s,\cdot)}^2_{L^2(\R)}ds\\
&\le \frac{\norm{\eta'}^2_{L^{\infty}(I_{T,2})}e^{2\gamma T}}{2}\norm{u_{0}}^2_{L^2(\R)}\le C(T).
\end{align*}
We have that
\begin{equation*}
\textrm{$\{\CL_{3,\delta}\}_{\delta>0}$ is bounded in $L^1_{loc}((0,T)\times\R)$.}
\end{equation*}
Let $K$ be a compact subset of $(0,T)\times\R$. By Lemmas \ref{lm:P-infty} and \ref{lm:linfty-u},
\begin{align*}
\norm{\gamma\eta'(\ued)\Ped}_{L^1(K)}&=\gamma\int_{K}\vert\eta'(\ue)\vert\vert\Pe\vert
dtdx\\
&\leq \gamma
\norm{\eta'}_{L^{\infty}(I_{T,2})}\norm{\Pe}_{L^{\infty}(I_{T,1})}\vert K \vert .
\end{align*}
Therefore, Murat's Lemma \cite{Murat:Hneg} implies that
\begin{equation}
\label{eq:GMC1}
    \text{$\left\{  \pt  \eta(\ued)+\px q(\ued)\right\}_{\delta>0}$
    lies in a compact subset of $\Hneg((0,\infty)\times\R)$.}
\end{equation}
The $L^{\infty}$ bound stated in Lemma \ref{lm:linfty-u}, \eqref{eq:GMC1} and the
 Tartar's compensated compactness method \cite{TartarI} give the existence of a subsequence
$\{\uedk\}_{k\in\N}$ and a limit function $u\in L^{\infty}((0,T)\times\R)$
such that
\begin{equation}\label{eq:convu}
    \textrm{$\uedk \to u$ a.e.~and in $L^{p}_{loc}((0,T)\times\R)$, $1\le p<\infty$}.
\end{equation}
Hence,
\begin{equation}
\label{eq:udelta-to-ueps}
 \textrm{$\uedk \to u$  in $L^{\infty}((0,T)\times\R)$}.
\end{equation}
Moreover, for convexity, we have
\begin{equation}
\label{eq:H-2-R}
\begin{split}
\norm{u(t,\cdot)}_{L^2(\R)}^2+2 e^{2\gamma t}\int_0^t e^{-2\gamma s}\norm{\px u(s,\cdot)}^2_{L^2(\R)}ds&\le C(T),\\
\norm{\px u(t,\cdot)}^2_{L^2(\R)} + \int_{0}^{t}\norm{\pxx u(s,\cdot)}^2_{L^2(\R)}ds&\le C(T),\\
\norm{\pxx u(t,\cdot)}^2_{L^2(\R)}+\int_{0}^{t}\norm{\pxxx u(s,\cdot)}^2_{L^2(\R)}ds&\le C(T).
\end{split}
\end{equation}
We need only to observe that
\begin{align*}
&2 e^{2\gamma t}\int_0^t e^{-2\gamma s}\norm{\px u(s,\cdot)}^2_{L^2(\R)}ds\\
&\quad\le 2 e^{2\gamma t}\liminf_{k}\int_0^t e^{-2\gamma s}\norm{\px \uedk(s,\cdot)}^2_{L^2(\R)}ds\le C(T),\\
&\int_{0}^{t}\norm{\pxx u(s,\cdot)}^2_{L^2(\R)}ds\le  \liminf_{k} \int_{0}^{t}\norm{\pxx\uedk(s,\cdot)}^2_{L^2(\R)}ds\le C(T),\\
&\int_{0}^{t}\norm{\pxxx u(s,\cdot)}^2_{L^2(\R)}ds\le \liminf_{k}\int_{0}^{t}\norm{\pxxx\uedk(s,\cdot)}^2_{L^2(\R)}ds\le C(T).
\end{align*}
Moreover, it follows from convexity and Lemma \ref{lm:38} that
\begin{equation}
\label{eq:1420}
\px^{\ell}u(t,\cdot)\in L^{2}(\R),\quad \ell >2, \quad t\in (0,T).
\end{equation}
Therefore, \eqref{eq:udelta-to-ueps}, \eqref{eq:H-2-R} and \eqref{eq:1420} give \eqref{eq:uePe1}. \eqref{eq:uePe2} follows from Lemma \ref{lm:P-infty}.

Finally, we prove that
\begin{equation}
\label{eq:pxp-u}
\int_{-\infty}^{x}u(t,y)dy=P(t,x), \quad \textrm{a.e.} \quad \textrm{in} \quad (t,x)\in I_{T,1}.
\end{equation}
Integrating the second equation of \eqref{eq:OHepsw1} on $(-\infty,x)$, for \eqref{eq:P-pxP-intfy1}, we have that
\begin{equation}
\label{eq:00113}
\int_{-\infty}^{x}\uedk(t,y)dy= \Pedk(t,x) -\dk\px\Pedk(t,x).
\end{equation}
We show that
\begin{equation}
\label{eq:px-1}
\textrm{$\delta\px\Ped(t,x)\to 0$ in $L^{\infty}((0,T)\times\R)$, $T>0$ as $\delta\to0$.}
\end{equation}
It follows from \eqref{eq:h2-P} that
\begin{equation*}
\delta\norm{\px\Ped}_{L^{\infty}((0,T)\times\R)}\leq \sqrt{\delta}e^{\gamma t}\norm{u_{\eps,0}}_{L^2(\R)} =\sqrt{\delta}C(T)\to 0,
\end{equation*}
that is \eqref{eq:px-1}.\\
Therefore, \eqref{eq:pxp-u} follows from \eqref{eq:uePe1}, \eqref{eq:uePe2}, \eqref{eq:00113} and \eqref{eq:px-1}.
The proof is done.
\end{proof}
\begin{lemma}\label{lm:u-null}
Let $u(t,x)$ be a classical solution of \eqref{eq:OHw-u}, or \eqref{eq:OHw}. Then,
\begin{equation}
\label{eq:con-u}
\int_{\R}u(t,x) dx=0, \quad t\ge 0,
\end{equation}
\end{lemma}
\begin{proof}
Differentiating \eqref{eq:OHw} with respect to $x$, we have
\begin{equation}
\label{eq:diff-oh}
\px(\pt u + \px f(u) -  \pxx u)=\gamma u.
\end{equation}
Since $u$ is a smooth solution of \eqref{eq:OHw}, an integration over $\R$ gives \eqref{eq:con-u}.
\end{proof}
We are ready for the proof of Theorem \ref{th:wellp}.
\begin{proof}[Proof of Theorem \ref{th:wellp}]
Lemma \ref{lm:exist} gives the existence of a classical solution of \eqref{eq:OHw-u}, or \eqref{eq:OHw}, while Lemma \ref{lm:u-null} says that
the solution has zero mean.

Let us show that $u(t,x)$ is unique and \eqref{eq:l2-stability} holds.
Let $u,\,v$ be two classical solutions of \eqref{eq:OHw-u}, or \eqref{eq:OHw}, that is
\begin{align*}
& \begin{cases}
\pt u+ f'(u)\px u=\gamma P^{u}+\pxx u,& t>0,  x\in\R,\\
\px P^{u}=u, & t>0, x\in\R,\\
u(0,x)=u_{0}(x),& x\in\R,
\end{cases}\\
&\begin{cases}
\pt v+f'(v)\px v=\gamma P^{v} +\pxx v, & t>0, x\in\R,\\
\px P^{v}=v, & t>0, x\in\R,\\
v(0,x)=v_{0}(x),& x\in\R.
\end{cases}
\end{align*}
Then, the function
\begin{equation}
\label{eq:def-di-omega}
\omega(t,x)=u(t,x)-v(t,x)
\end{equation}
is solution of the following Cauchy problem
\begin{equation}
\label{eq:epsw}
\begin{cases}
\pt \omega+ f'(u)\px u -f'(v)\px v =\gamma\Omega+ \pxx\omega,&\quad t>0,\ x\in\R,\\
\px\Omega=\omega,&\quad t>0,\ x\in\R,\\
\omega(0,x)=u_{0}(x) - v_{0}(x) ,&\quad x\in\R,
\end{cases}
\end{equation}
where
\begin{equation}
\label{eq:def-di-Omega}
\begin{split}
\Omega(t,x)&=P^{u}(t,x)- P^{v}(t,x)\\
&=\int_{-\infty}^{x} u(t,y)dy - \int_{-\infty}^{x} v(t,y)dy\\
&=\int_{-\infty}^{x} (u(t,y)-v(t,y))dy =\int_{-\infty}^x\omega(t,y)dy.
\end{split}
\end{equation}
It follows from Lemma \ref{lm:u-null} and \eqref{eq:def-di-Omega} that
\begin{equation}
\label{eq:Omega-in-infty}
\Omega(t,\infty)= \int_{\R} u(t,y)dy - \int_{\R} v(t,y)dy=0.
\end{equation}
Observe that, from \eqref{eq:def-di-omega},
\begin{align*}
f'(u)\px u-f'(v)\px v &= f'(u)\px u - f'(u)\px v + f'(u)\px v - f'(v)\px v\\
&=f'(u)\px(u-v) +( f'(u)- f'(v))\px v\\
&=f'(u)\px\omega + ( f'(u)- f'(v))\px v.
\end{align*}
Therefore, the first equation of \eqref{eq:epsw} is equivalent to the following one:
\begin{equation}
\label{eq:epsw1}
\pt \omega+ f'(u)\px\omega + ( f'(u)- f'(v))\px v =\gamma\Omega+ \pxx\omega.
\end{equation}
Moreover, since $u$ and $v$ are in $L^{\infty}((0,T)\times\R)$, we have
that
\begin{equation}
\label{eq:f1}
\Big\vert f'(u(t,x))- f'(v(t,x))\Big\vert \leq C(T) \vert u(t,x) - v(t,x)\vert,\quad (t,x)\in (0,T)\times\R,
\end{equation}
where
\begin{equation}
\label{eq:sup}
C(T)=\sup_{(0,T)\times\R}\Big\{\vert f''(u)\vert + \vert f''(v)\vert\Big\}.
\end{equation}
Therefore, \eqref{eq:def-di-omega} and \eqref{eq:f1} give
\begin{equation}
\label{eq:f1omega}
\Big\vert f'(u(t,x))- f'(v(t,x))\Big\vert \leq C(T) \vert\omega(t,x)\vert,\quad (t,x)\in (0,T)\times\R.
\end{equation}
Multiplying \eqref{eq:epsw1} by $\omega$, an integration on $\R$ gives
\begin{align*}
\frac{d}{dt}\int_{\R} \omega^2dx=&2\int_{\R} \omega\pt\omega dx\\
=&2\int_{\R}\omega\pxx\omega dx-2\int_{\R}\omega f'(u)\px\omega  dx\\
&-2\int_{\R}\omega( f'(u)- f'(v))\px v dx  +2\gamma\int_{\R}\Omega\omega dx\\
=& -2\int_{\R}(\px\omega)^2 dx +\int_{\R}\omega^2 f''(u)\px u dx\\
&-2\int_{\R}\oeps( f'(u)- f'(v))\px v dx  +2\gamma\int_{\R}\Omega\omega  dx.
\end{align*}
It follows from the second equation of \eqref{eq:epsw} and Lemma \ref{lm:u-null} that
\begin{equation}
\label{eq:1230}
\begin{split}
&\frac{d}{dt}\norm{\omega(t,\cdot)}^2_{L^{2}(\R)}+2\norm{\px\omega(t,\cdot)}^2_{L^{2}(\R)}\\
&\quad\le \int_{\R}\omega^2 \vert f''(u)\vert\vert\px u\vert dx + 2\int_{\R}\vert\omega\vert\vert( f'(u)- f'(v))\vert\vert\px v\vert dx.
\end{split}
\end{equation}
Since $u(t,\cdot),\,v(t,\cdot)\in H^{\ell}(\R),\ell >2$, for each $t\in (0,T)$, then
\begin{equation}
\label{eq:1231}
\px u(t,\cdot),\px v(t,\cdot)\in H^{\ell-1}(\R)\subset L^{\infty}(\R), \quad t\in (0,T).
\end{equation}
Therefore, thanks to \eqref{eq:f1}, \eqref{eq:sup}, \eqref{eq:1230} and \eqref{eq:1231},
\begin{align*}
\frac{d}{dt}\norm{\omega(t,\cdot)}^2_{L^{2}(\R)}+2\norm{\px\omega(t,\cdot)}^2_{L^{2}(\R)}\le C(T)\norm{\omega(t,\cdot)}^2_{L^{2}(\R)}.
\end{align*}
The Gronwall Lemma gives
\begin{equation}
\label{eq:gro1}
\norm{\omega(t,\cdot)}^2_{L^{2}(\R)}+2 e^{C(T)t}\int_{0}^{s}e^{-C(T)s} \norm{\px\omega(s,\cdot)}^2_{L^{2}(\R)} ds \le e^{C(T)t}\norm{\omega_{0}}^2_{L^2(\R)}.
\end{equation}
Hence, \eqref{eq:l2-stability} follows from \eqref{eq:def-di-omega}, \eqref{eq:epsw} and \eqref{eq:gro1}.
\end{proof}


\begin{thebibliography}{54}


\bibitem{B}
{\sc J. Boyd.}
\newblock Ostrovsky and HunterÕs generic wave equation for weakly dispersive waves: matched
asymptotic and pseudospectral study of the paraboloidal travelling waves (corner and near-corner waves).
\newblock {\em Euro. Jnl. of Appl. Math.}, 16(1):65--81, 2005.

\bibitem{CG}
{\sc J. Canosa and J. Gazdag.}
\newblock The Korteweg-de Vries-Burgers equation
\newblock {\em Journal of Computational Physics}, vol. 23, no. 4, 393-403, 1977.

\bibitem{Cd}
{\sc G.~M. Coclite and L. di Ruvo.}
\newblock Wellposedness of bounded solutions of the non-homogeneous initial boundary value problem for the Ostrovsky-Hunter equation.
\newblock {To appear on \em {J. Hyperbolic Differ. Equ.}}

\bibitem{Cd1}
{\sc G.~M. Coclite and L. di Ruvo.}
\newblock  Wellposedness results for the Short Pulse Equation.
\newblock To appear on {\em Z. Angew. Math. Phys.}




\bibitem{CHK:ParEll}
{\sc G.~M. Coclite, H.~Holden, and K.~H. Karlsen.}
\newblock Wellposedness for a parabolic-elliptic system.
\newblock {\em Discrete Contin. Dyn. Syst.}, 13(3):659--682, 2005.



\bibitem{GP}
{\sc R. Grimshaw and D. E. Pelinovsky.}
\newblock Global existence of small-norm solutions in the reduced Ostrovsky equation.
  \newblock {\em Discr. Cont. Dynam. Syst. A},  34:557--566, 2014.

\bibitem{dR}
{\sc L. di Ruvo.}
\newblock Discontinuous solutions for the Ostrovsky--Hunter equation and two phase flows.
\newblock {\em Phd Thesis, University of Bari}, 2013.
\newblock{www.dm.uniba.it/home/dottorato/dottorato/tesi/}.


\bibitem{GL}
{\sc G. Gui and Y. Liu.}
\newblock On the Cauchy problem for the Ostrovsky equation with positive dispersion.
\newblock {\em Comm. Part. Diff. Eqs.}, 32(10-12):1895--1916, 2007.


\bibitem{FM}
{\sc Z.S Fend and Q.g. Meng}
\newblock Burgers-Korteweg-de Vries equation and its traveling solitary waves.
\newblock{\em S. in China Series A: Mathem. Springer-Verlag}, 50(3):412--422, 2007.

\bibitem{H}
{\sc J. Hunter.}
\newblock Numerical solutions of some nonlinear dispersive wave equations.
Computational solution of nonlinear systems of equations (Fort Collins, CO, 1988)
\newblock {\em Lectures in Appl. Math.}, 26, Amer. Math. Soc., Providence, RI, 301--316, 1990.


\bibitem{HT}
{\sc J. Hunter and K. P. Tan.}
\newblock Weakly dispersive short waves
\newblock {\em Proceedings of the IVth international Congress on Waves and
 Stability in Continuous Media}, Sicily, 1987.

\bibitem{KdV}
{\sc D. J. Korteweg, and G. de Vries.}
\newblock{On the change of form of long waves advancing in a rectangular canal, and on a new type of long stationary waves.}
\newblock{\em Philosophical Magazine}, vol. 39, no. 240, 422-443, 1895.




\bibitem{LL}
{\sc S. Levandosky and Y. Liu.}
\newblock Stability of solitary waves of a generalized Ostrovsky equation.
\newblock {\em SIAM J. Math. Anal.}, 38(3):985--1011, 2006.

 \bibitem{LL:07}
{\sc S. Levandosky and Y. Liu.}
\newblock Stability and weak rotation limit of solitary waves of the Ostrovsky equation.
\newblock {\em Discr. Cont. Dyn. Syst. B}, 7(7):793--806, 2007.

\bibitem{LM}
{\sc F. Linares and A. Milanes.}
\newblock Local and global well-posedness for the Ostrovsky equation.
\newblock {\em J. Diff. Eqs.}, 222(2):325--340, 2006.

\bibitem{L}
{\sc Y. Liu.}
\newblock On the stability of solitary waves for the Ostrovsky equation.
\newblock {\em Quart. Appl. Math.}, 65(3):571--589, 2007.

\bibitem{LPS}
{\sc Y. Liu, D. Pelinovsky, and A. Sakovich.}
\newblock Wave breaking in the Ostrovsky--Hunter equation.
  \newblock  {\em SIAM J. Math. Anal.} 42(5):1967--1985, 2010.

  \bibitem{LV}
{\sc Y. Liu and  V. Varlamov.}
\newblock Cauchy problem for the Ostrovsky equation.
\newblock {\em Discr. Cont. Dyn. Syst.}, 10(3):731--753, 2004.

  \bibitem{LV:JDE}
{\sc Y. Liu and  V. Varlamov.}
\newblock Stability of solitary waves and weak rotation limit for the Ostrovsky equation.
\newblock {\em J. Diff. Eqs.}, 203(1):159--183, 2004.

 \bibitem{MPV}
{\sc A. J. Morrison, E. J. Parkes, and V. O. Vakhnenko.}
\newblock The $N$ loop soliton solutions of the Vakhnenko equation.
\newblock {\em Nonlinearity}, 12(5):1427--1437, 1999.

 \bibitem{Murat:Hneg}
{\sc F.~Murat.}
\newblock L'injection du c\^one positif de ${H}\sp{-1}$\ dans ${W}\sp{-1,\,q}$\  est compacte pour tout $q<2$.
\newblock {\em J. Math. Pures Appl. (9)}, 60(3):309--322, 1981.


\bibitem{O}
{\sc L. A. Ostrovsky.}
\newblock Nonlinear internal waves in a rotating ocean.
\newblock {\em Okeanologia}, 18:181--191, 1978.

\bibitem{P}
{\sc E. J. Parkes.}
\newblock Explicit solutions of the reduced Ostrovsky equation.
\newblock {\em Chaos, Solitons and Fractals}, 31(3):602--610, 2007.

\bibitem{PV}
{\sc E. J. Parkes and V. O. Vakhnenko.}
\newblock The calculation of multi-soliton solutions of the Vakhnenko equation by the inverse scattering method.
\newblock {\em Chaos, Solitons and Fractals}, 13(9):1819--1826, 2002.





\bibitem{S}
{\sc Y. A. Stepanyants.}
\newblock On stationary solutions of the reduced Ostrovsky equation: periodic waves, compactons and compound solitons.
\newblock {\em Chaos, Solitons and Fractals}, 28(1):193--204, 2006.


\bibitem{Su}
{\sc J.J. Shu}
\newblock {The Proper analytical solution of the Korteweg-de Vries-Burgers equation.}
\newblock{ \em J. of Physics A-Mathem. and General}, 20(2):49-56, 1987.

\bibitem{TartarI}
L.~Tartar.
\newblock Compensated compactness and applications to partial differential   equations.
\newblock In {\em Nonlinear analysis and mechanics: Heriot-Watt Symposium, Vol.  IV},
pages 136--212. Pitman, Boston, Mass., 1979.

\bibitem{T}
{\sc  K. Tsugawa.}
\newblock Well-posedness and weak rotation limit for the Ostrovsky equation.
\newblock {\em J. Differential Equations} 247(12):3163--3180, 2009.

\end{thebibliography}
\end{document}